\documentclass{article}

\newcommand{\ev}[1]{#1}%
\newcommand{\sv}[1]{}%
\newcommand{\lv}[1]{#1}%
\usepackage{etoolbox}
\newcommand{\appendixText}{}

 \newcommand{\toappendix}[1]{#1}%

\sv{\hideLIPIcs}  %

\usepackage[T1]{fontenc}
\usepackage[utf8]{inputenc}
\usepackage{graphicx} %
\usepackage{hyperref}
\usepackage{amsmath}
\usepackage{amsthm}
\usepackage{thmtools}
\usepackage{thm-restate}
\usepackage{amssymb}
\lv{\usepackage{fullpage}}
\usepackage{complexity}
\usepackage[dvipsnames]{xcolor}
\lv{
\usepackage[capitalize,noabbrev]{cleveref}}
\lv{\usepackage{subcaption}}
\usepackage{mathtools}
 
\usepackage{tikz}
\usetikzlibrary{shapes.geometric}
\usetikzlibrary{positioning}
\usetikzlibrary{math,calc}
\usetikzlibrary{decorations.pathmorphing,decorations.pathreplacing}
\usetikzlibrary{backgrounds}

\lv{
\usepackage{authblk}
  \author[1]{Gaspard Charvy}%
  \author[2]{Tomáš Masařík%
  \thanks{T.M.~was supported by Polish National Science Centre SONATA-17 grant number 2021/43/D/ST6/03312.}}

  \affil[1]{Univ Rennes, F-35000 Rennes, France

  \texttt{gcharvy@ens-rennes.fr}} %
  
  \affil[2]{Institute of Informatics, Faculty of Mathematics, Informatics and Mechanics, University~of~Warsaw, Warszawa, Poland 

  \texttt{masarik@mimuw.edu.pl}}
  \date{}

\newtheorem{theorem}{Theorem}[section]

\newtheorem{lemma}[theorem]{Lemma}
\newtheorem{observation}[theorem]{Observation}

\newtheorem{conjecture}[theorem]{Conjecture}
}
\newtheorem{question}[theorem]{Question}

\sv{

\authorrunning{G. Charvy and T. Masařík} %

\Copyright{Gaspard Charvy and Tomáš Masařík} %



\GAIDecl{
  Generative AI was used only to polish the completed manuscript, which was written by the authors without any AI assistance.
}

\category{Track 1, Long Paper} %

}

\let\epsilon\varepsilon
\newcommand{\unc}{\ensuremath{\mathrm{unc}}}
\newcommand{\oh}{o}
\newcommand{\Dd}{\mathcal{D}}

\title{General Strong Bound on the Uncrossed Number via a Tight Bound for the Maximum Uncrossed Subgraph Number}
\sv{\titlerunning{General Strong Bound on the Uncrossed Number}}

\begin{document}

\maketitle

\begin{abstract}
 We investigate a very recent concept for visualizing various aspects of a graph in the plane using a collection of drawings introduced by Hliněný and Masařík~[GD 2023].
Formally, given a graph $G$, we aim to find an \emph{uncrossed collection} containing drawings of $G$ in the plane such that each edge of $G$ is uncrossed in at least one drawing in the collection.
The \emph{uncrossed number} of $G$ ($\unc(G)$) is the smallest integer $k$ such that an uncrossed collection for $G$ of size $k$ exists. 
The uncrossed number is lower-bounded by the well-known \emph{thickness}, which is an edge-decomposition of $G$ into planar graphs.
Hence, this connection gives a trivial lower-bound $\left\lceil\frac{|E(G)|}{3|V(G)|-6}\right\rceil  \le \unc(G)$.
Balko, Hliněný, Masařík, Orthaber, Vogtenhuber, and Wagner [GD~2024] presented the first non-trivial and general lower-bound on the uncrossed number.
While the exact formulation is a bit cumbersome, its biggest impact can be summarized in terms of dense graphs (where $|E(G)|=\epsilon(|V(G)|)^2$ for some $\epsilon>0$).
Their lower-bound states $\left\lceil\frac{|E(G)|}{c_\epsilon|V(G)|}\right\rceil  \le \unc(G)$, where $c_\epsilon\ge 2.82$ is a constant depending on~$\epsilon$.

We improve the lower-bound to ${\left\lceil\frac{|E(G)|}{3|V(G)|-6-\sqrt{2|E(G)|}+\sqrt{6(|V(G)|-2)}}\right\rceil  \le \unc(G)}$.
Translated to dense graphs regime, the bound yields a multiplicative constant $c'_\epsilon=3-\sqrt{2\epsilon}$ in the following expression ${\left\lceil\frac{|E(G)|}{c'_\epsilon|V(G)|+\oh(|V(G)|)}\right\rceil  \le \unc(G)}$.
Therefore, it is tight (up to lower-order terms) for $\epsilon \approx \frac{1}{2}$ as warranted by complete graphs.

In fact, we formulate our result in the language of the \emph{maximum uncrossed subgraph number}, %
that is, the maximum number of edges of $G$ that are not crossed (uncrossed) in a drawing of $G$ in the plane.
In that case, we also provide a construction certifying that our bound is asymptotically tight (up to lower-order terms) on dense graphs for all $\epsilon>0$.

\lv{%
  \bigskip{} \begin{center} 
\end{center}
}
\end{abstract}

\clearpage

\section{Introduction}

Graph visualization is one of the fundamental graph problems and was among the first problems shaping the field of graph theory.
In this paper, we investigate problems that aim to draw a graph in the plane nicely using a standard definition of a (plane)\footnote{We only consider drawings in the plane in this paper. Hence, for conciseness, we often say just drawing in most cases, but we always mean a plane drawing.} drawing; see~\cref{sec:prelim} for a formalization.
There have been plenty of interesting definitions of what drawing a graph nicely in the plane means.
Perhaps the oldest concept was one controlling crossings of the edges.
A \emph{crossing} in a drawing $D$ of $G$ is a common interior point of two distinct edges of $D$. %
As a related notion, we say that an edge $e$ of $D$ is \emph{uncrossed} in $D$ if it does not share a crossing with any other edge of $D$.
In this direction, the \emph{crossing number} ($\mathrm{cr}(G)$) of a graph $G$ is defined as the smallest number of crossings among all possible drawings of $G$. 
We refer to the book and the recently updated survey by Schaefer~\cite{Schaefer18, Schaefer13} for further details about this interesting notion.

Our primary focus will be on the concept of the uncrossed edges.
Instead of minimizing the number of individual crossings in the whole graph, we rather maximize the number of edges that are uncrossed among all possible drawings of $G$.
This number is called the \emph{maximum uncrossed subgraph number}\footnote{Do not confuse this notion with the \emph{maximum planar subgraph number}, which is a complementary notion to the \emph{skewness} of a graph. The skewness is defined as the smallest number of edges that need to be removed from a graph to make it planar.} of $G$ ($h(G)$).
Alternatively, its complement is called \emph{edge crossing number} in the literature, i.e., minimizing the number of edges that are crossed among all possible drawings of $G$.
This problem formulation dates back to 1964 when Ringel established $h(K_n)=2n-2$~\cite{ringel64}.
Since then, there has been some interest in the problem, but it has received considerably less attention compared to the crossing number.

Another early direction dating back to 1961~\cite{Har61} was a study of the \emph{thickness} of graph $G$ ($\theta(G)$), which is the smallest number of planar graphs into which we can edge-decompose $G$.
Much less studied but also very relevant for our main problem is a related notion of the \emph{outerthickness} of graph $G$ ($\theta_o(G)$), which is the smallest number of outerplanar graphs into which we can edge-decompose $G$.
For related results covering those notions, we refer to the 1998 survey by Mutzel, Odenthal, and Scharbrodt~\cite{ThickSurvey}, and more recent papers~\cite{ThicknessRecent,GeomThick,OuterthicknessRecent} pointing at the current development.

In this paper, the main object of our study is the uncrossed number defined only very recently by Hliněný and Masařík~\cite{masHlin23}.
Their idea is to visualize a graph not by a single drawing but rather using a collection of drawings such that each aspect of the graph is highlighted in at least one drawing nicely.
Such a general statement can, of course, have many formalizations, as what is the meaning of aspect and nicely needs to be specified.
Hliněný and Masařík proposed to start with well-established concepts used to characterize the quality of a single drawing, such as the crossing number, the maximum uncrossed subgraph number, and the thickness.
They define an \emph{uncrossed collection} of $G$ as a collection $\mathcal{D}$ containing drawings of $G$ such that for each edge $e\in E(G)$ there is a drawing $D\in \mathcal{D}$ such that $e$ is uncrossed in $D$.
Then the \emph{uncrossed number} of $G$ ($\unc(G)$) is the least integer $k$ for which an uncrossed collection of $G$ of size $k$ exists.
It is easy to observe that the uncrossed edges of each $D\in\Dd$ have to form a planar graph.
Likewise, if we have an edge decomposition of $G$ into outerplanar graphs, then for each such outerplanar subgraph, we can always realize all the remaining edges in the outer face without crossing any of the outerplanar edges.
This yields a trivial uncrossed collection of $G$.
Combining those observations with a result of Gon\c{c}alves~\cite{Goncalves2}, which states $\theta_o(G)\leq2\theta(G)$, we obtain the following relations:
\begin{equation}
  \theta(G) \leq \unc(G) \leq \theta_o(G) \leq 2 \theta(G).
\end{equation}

Recently, Balko, Hliněný, Masařík, Orthaber, Vogtenhuber, and Wagner~\cite{Uncrossed2} proved the exact bounds on the uncrossed number for complete and complete bipartite graphs.
Moreover, they showed the first non-trivial lower-bound.
\begin{theorem}[{\cite[Theorem 3]{Uncrossed2}}]
\label{thm:old-lb}
Every connected graph $G$ with $n \geq 3$ vertices and $m \geq 0$ edges satisfies
\[
\unc(G) \geq \left\lceil\frac{m}{(3n-5+\sqrt{(3n-5)^2-4m})/2}\right\rceil.
\]
\end{theorem}
Despite being a general bound for all graphs, it is strongest in the dense regime, i.e., for graphs such that $m=\epsilon n^2$ for some $\epsilon>0$.
Then \cref{thm:old-lb} states $\unc(G) \ge \left\lceil\frac{m}{c_\epsilon n}\right\rceil  $, where $c_\epsilon\ge 2.82$ is a constant depending on $\epsilon$.

\subparagraph*{Maximum Uncrossed Subgraph Number.}
Surprisingly, very little is known about the Maximum Uncrossed Subgraph Number.
Ringel~\cite{ringel64} introduced the notion and gave the exact bound for complete graphs.
Mengersen~\cite{germanPaper} extended Ringel’s result from complete to complete multipartite graphs.
\begin{theorem}[Exact $h$ value for complete~\cite{ringel64} and complete bipartite~\cite{germanPaper} graphs]
\label{thm-germanPaper}
For all positive integers $n$, we have $h(K_n)=2n-2$.
For all positive integers $n_1$ and $n_2$ with $n_1 \leq n_2$, we have
\[
h(K_{n_1,n_2}) = 
\begin{cases}
    2n_1+n_2-2, & \text{for } n_1=n_2 \\
    2n_1+n_2-1, & \text{for } n_1 < n_2 < 2n_1 \\
    2n_1+n_2, & \text{for } 2n_1 \leq n_2.
\end{cases}
\]
\end{theorem}
\noindent
The \NP-completeness of the problem was shown recently in~\cite[Corollary 5]{Uncrossed2}.
Later, Colin de Verdière and Hliněný~\cite[Proposition 5.9]{ecrFPT} complemented this hardness result with an \FPT{} algorithm parameterized by the number of crossed edges.
This also gives fixed-parameter tractability with respect to $h(G)$: indeed, every connected graph has $h(G)\ge n-1$, and hence the number of crossed edges is at most $\binom{n}{2}-h(G)=O(h(G)^2)$.
In fact, their algorithm can be generalized to any surface.
In~\cite[Theorem~18]{BeyondPlanar24}, the authors gave an example of a graph that has $\Omega(k^2)$ crossings when restricted to drawings that allow at most $k$ crossed edges while the unconstrained crossing number of that graph is only $O(k)$.

The maximum uncrossed subgraph notion has been generalized in various directions.
We list here a non-exhaustive collection of major generalizations; see also~\cite{Schaefer13} for further pointers.
Harborth and Mengersen~\cite{haMen74} also studied drawings where they attempted to construct drawings with a prescribed number of uncrossed edges (hence not necessarily optimal).
The same authors~\cite{haMen90} further generalized the notion when they count edges having at most a fixed constant number of crossings as uncrossed.
Harborth and Thürmann~\cite{haThu96} studied a variant of the notion subject to rectilinear drawings.
This direction was very recently picked up by Chen and Solé-Pi~\cite{ChenSole-Pi} who studied the whole sequence characterizing the number of crossings per edge in the rectilinear setting.
Bannister, Eppstein, and Simons~\cite{Bookecr} studied a book embedding variant.
In a different direction, Kynčl and Valtr~\cite{KV} studied an interesting concept. They estimated the number of crossings of an edge with the least number of crossings over all drawings of a complete graph where all pairs of edges intersect at most once.

\subparagraph*{Uncrossed Collection Related Results.}
Hliněný and Masařík~\cite{masHlin23} also suggested another optimization variant.
Instead of bounding the number of drawings in an uncrossed collection, they proposed counting the sum of all crossings that appeared in all drawings of some uncrossed collection.
They denoted the smallest such number as the \emph{uncrossed crossing number}.
Recently, another variant of the collection concept has appeared in the literature.
The authors~\cite{unbent} defined the \emph{unbent collection} of an $(\le{}\!4)$-regular planar graph as a collection of orthogonal drawings such that for each edge there exists at least one drawing in the collection where the edge is not bent.
Analogously, the \emph{unbent number} is the minimum size of an unbent collection, and the \emph{unbent bent number} is the minimum number of bends in all drawings within an unbent collection.
The above results add to a recent line of work in graph visualization, which uses a collection or a sequence of drawings.
The general idea is to slice a given graph $G$ into a sequence of planar drawings where each vertex (or edge) appears only in consecutive drawings within the sequence; see the following references for precise formal definitions. 
The intention is that the viewer reconstructs the whole structure by watching the sequence.
Probably the oldest such concept was edge-oriented \emph{streamed planarity} where the number of drawings in which each edge is visible is fixed~\cite{Streamed}.
A related idea, graph \emph{stories}, focuses on vertices still using a fixed-size window of appearances~\cite{Stories, Stories2}.
A more general \emph{storyplan model} allows for a variable number of appearances of vertices in the sequence of drawings~\cite{Storyplan,Storyplan2, Storyplan3}.

\subsection{Our results}
We now present our main result, a strong general lower-bound on the uncrossed number.

\begin{restatable}[Uncrossed Number Lower-bound]{theorem}{thmunc}
\label{thm:unc}
    For every $n\geq 3$, for every connected graph $G$ with $n$ vertices and $m$ edges, we have
    \begin{align}
        \mathrm{unc}(G) \geq \left\lceil\frac{m}{3n-6 -\sqrt{2m} +\sqrt{6(n-2)}}\right\rceil.
    \end{align}
\end{restatable}

In particular, for dense graphs with $n$ vertices and $m = \varepsilon n^2$ edges, we have:

\begin{align}\label{eq:dense}
    \mathrm{unc}(G) \geq \left\lceil\frac{m}{\left(3 -\sqrt{2\varepsilon}\right)n + o(n)}\right\rceil = \left\lceil\frac{\varepsilon n}{3 -\sqrt{2\varepsilon} + o(1)}\right\rceil.
\end{align}

\begin{figure}
    \centering
    \includegraphics[width=0.65\linewidth]{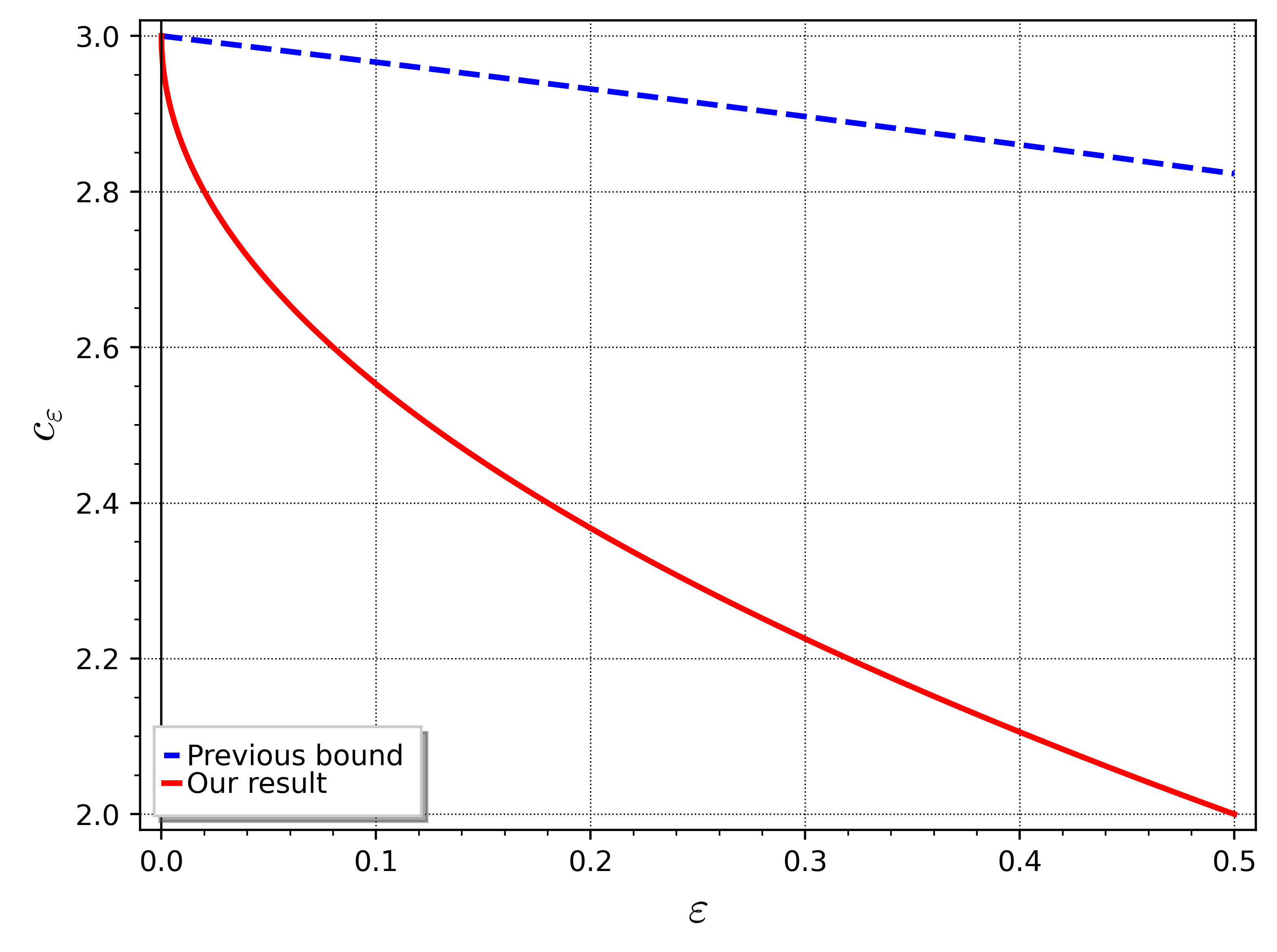}
    \caption{Multiplicative constant $c_\epsilon$ given by \cref{thm:old-lb} (\cite{Uncrossed2}) depicted by dashed blue line and \cref{eq:dense} drawn as solid red line. Note that our values of $c_\varepsilon$ are tight for $\varepsilon$ tending towards $\frac{1}{2}$ (complete graphs) and towards 0 (planar triangulated graphs).}
    \label{fig:asymptotical_ceps}
\end{figure}

We can directly compare \cref{eq:dense} with the previous general bound (\cref{thm:old-lb}), which gives the multiplicative constant $c_\epsilon = \frac{3 + \sqrt{9 - 4 \epsilon}}{2}$ (up to lower-order terms) $\ge 2.82$ (in the bound $\unc(G) \ge \left\lceil\frac{m}{c_\epsilon n}\right\rceil$) while \cref{thm:unc} (see \cref{eq:dense}) has $c_\epsilon=3-\sqrt{2\epsilon}$ (up to lower-order terms), which is substantially stronger.
\cref{fig:asymptotical_ceps} presents a graphic comparison of those multiplicative constants.
We remark that \cref{thm:unc} is asymptotically optimal in the regime when $\epsilon$ tends towards $\frac{1}{2}$.
Recall that Balko, Hliněný, Masařík, Orthaber, Vogtenhuber, and Wagner \cite[Theorem 1]{Uncrossed2} proved the exact bound on the uncrossed number of a complete graph.
For the case $n>7$ the bound is the following:
\begin{align}\label{eq:kn}
    \mathrm{unc}(K_n) = \left\lceil\frac{n-1}{4}\right\rceil.
\end{align}
\cref{eq:kn} shows that the bound in \cref{thm:unc} (see \cref{eq:dense}) is asymptotically tight for complete graphs, as $n\to\infty$, where $m/n^2\to\frac{1}{2}$. %

\ev{
As another consequence of our proof method, we formulate a tighter result for triangle-free graphs.

\begin{restatable}[Uncrossed Number Lower-bound for Triangle-Free Graphs]{theorem}{thmunctfree}
\label{thm:unc_tfree}
    For every $n\geq 3$, for every triangle-free connected graph $G$ with $n$ vertices and $m$ edges, we have
    \begin{align}
        \mathrm{unc}(G) \geq \left\lceil\frac{m}{2n-4 -\sqrt{\frac{m}{2}} +\sqrt{\frac{5}{2}(n-2)}}\right\rceil.
    \end{align}
\end{restatable}
}

We now observe a connection with the maximum uncrossed subgraph number, which was already exploited by~\cite{Uncrossed2}.
Indeed, the total number of uncrossed edges in a collection of $d$ drawings of $G$ is at most $h(G) \times d$. Then, by definition, as every edge of $G$ has to be uncrossed in at least one drawing of an uncrossed collection of drawings of $G$, we have:
\begin{observation}[Uncrossed Number--Maximum Uncrossed Subgraph Connection]
\label{obs:uncH}
\begin{align*}
    \mathrm{unc}(G) \geq \left\lceil \frac{m}{h(G)}  \right \rceil.
\end{align*}
\end{observation}

In fact, we formulate our main theorem in a stronger form in terms of the maximum uncrossed subgraph number.

\begin{restatable}[Maximum Uncrossed Subgraph Number Upper Bound]{theorem}{thmh}
\label{thm:hG}
    For every $n \geq 3$, for every connected graph $G$ with $n$ vertices and $m$ edges, we have:
    \begin{align}
        h(G) \leq 3n - 6 - \sqrt{2m} + \sqrt{6(n-2)}.
        \end{align}
\end{restatable}

We can readily see that \cref{thm:unc} is a consequence of \cref{thm:hG} using \cref{obs:uncH}.
It is easy to observe that the bound in \cref{thm:hG} is tight for triangulated planar graphs (with $n$ vertices and $3n-6$ edges).
We show an asymptotically tight (up to lower-order terms) construction for graphs of any density.
We formalize this claim in the following theorem.

\begin{restatable}[Maximum Uncrossed Subgraph Tight Construction]{theorem}{thmtight}
\label{thm:tightness}
For any density $0 < \varepsilon \le \frac{1}{2}-\frac{1}{2n}$ and for any $n\ge \frac{3}{\epsilon}$, there exists a connected graph $G$ on $n$ vertices and $m$ edges such that:
\begin{enumerate}
    \item[1)]\phantomsection \label{c:1} $h(G) \geq 3n-3-\sqrt{2m}$, and
    \item[2)]\phantomsection\label{c:2} $\varepsilon \le \frac{m}{n^2} \le \varepsilon + \frac{1}{n} -\frac{3}{n^2}$.
\end{enumerate}
\end{restatable}

Note that the maximum value of $\epsilon=\frac{1}{2}-\frac{1}{2n}$ in \cref{thm:tightness} is the density of a complete graph.
From Property \hyperref[c:1]{1)}, we can see that for the constructed graph $G$, the difference between $h(G)$ and $3n-6-\sqrt{2m} + \sqrt{6(n-2)}$ is at most $\sqrt{6(n-2)}-3 = o(n)$, so the bound in \cref{thm:hG} is tight up to lower-order terms for such graphs.

In the specific case of triangle-free graphs, we can slightly change the proof from \cref{sec:main} to obtain a better bound for the maximum uncrossed subgraph number:

\begin{restatable}[Maximum Uncrossed Subgraph Number Upper Bound for Triangle-Free Graphs]{theorem}{thmhtfree}
\label{thm:hG-triangle}
    For every $n \geq 3$, for every triangle-free connected graph $G$ with $n$ vertices and $m$ edges, we have:
    \begin{align}
        h(G) \leq 2n - 4 - \sqrt{\frac{m}{2}} + \sqrt{\frac{5}{2}(n-2)}.
        \end{align}
\end{restatable}

Unlike \cref{thm:hG}, this result does not appear to be tight.
For instance, for even $n$, our bound states that $h(K_{\frac{n}{2},\frac{n}{2}}) \leq \left(2-\frac{1}{2\sqrt{2}}\right)n - 4 + \sqrt{\frac{5}{2}(n-2)}$.
However, according to \cref{thm-germanPaper}, $h(K_{\frac{n}{2},\frac{n}{2}}) = \frac{3}{2}n - 2$, which is an asymptotic gap of around $0.14n$. 

According to Mantel's theorem (1907)~\cite{Mantel07}, the maximum number of edges in a triangle-free graph with $n$ vertices is $m = \left\lfloor \frac{n^2}{4} \right\rfloor$ edges, and it is attained by a complete bipartite graph $K_{\frac{n}{2},\frac{n}{2}}$ or $K_{\frac{n-1}{2},\frac{n+1}{2}}$ (depending on the parity of $n$).
This means that \cref{thm:hG-triangle} is not tight for the densest possible triangle-free graphs, and possibly not for other graph densities either.
Nevertheless, similarly to \cref{thm:tightness}, we provide a construction giving a large maximum uncrossed subgraph for every feasible density in the triangle-free setting:
\begin{restatable}[Maximum Uncrossed Subgraph Construction for Triangle-Free Graphs]{theorem}{thmtighttfree}
\label{thm:tightness_tfree}
For any density $0 < \varepsilon \le \frac{1}{4}$ and for any $n \ge \frac{2}{\varepsilon}$, there exists a connected triangle-free graph $G$ on $n$ vertices and $m$ edges such that:
\begin{enumerate}
    \item[1)]\phantomsection \label{ctfree:1} $h(G) \geq 2n-2-\sqrt{m}$, and
    \item[2)]\phantomsection\label{ctfree:2} $\varepsilon - \frac{1}{n} + \frac{5}{n^2} \leq \frac{m}{n^2} \leq \varepsilon$.
\end{enumerate}
\end{restatable}

\subparagraph*{Organization of the Paper.}
We prove \cref{thm:hG} in \cref{sec:main} and the respective construction (\cref{thm:tightness}) in \cref{sec:construct}.
\lv{We prove both improved statements for triangle-free graphs: \cref{thm:hG-triangle} and the respective construction (\cref{thm:tightness_tfree}) in \cref{sec:triangle}.}\sv{We prove \cref{thm:hG-triangle} modulo two technical lemmas deferred to the appendix, and give the construction for \cref{thm:tightness_tfree}, whose proof is deferred to the appendix, in \cref{sec:triangle}.}
We end this paper with a short conclusion (\cref{sec:conc}).

\subsection{Preliminaries}\label{sec:prelim}
We work with simple, undirected, connected graphs throughout the paper. For $n\ge 4$, let $W_n$ be the wheel graph on $n$ vertices, that is, the graph obtained from a cycle on $n-1$ vertices by adding a universal vertex. See \cref{fig-W6} for an illustration.

\subparagraph*{Drawings in the Plane.}
In this paragraph, we give the formal details of drawings of graphs in the plane.
We follow a standard definition of a \emph{(plane) drawing} of a graph $G$:
The vertices are represented by distinct points in the plane, and each edge corresponds to a simple continuous arc connecting the images of its end-vertices.
As usual, we identify the vertices and their images, as well as the edges and the arcs representing them.
We require that the edges pass through no vertices other than their endpoints. 
Moreover, two edges sharing an endpoint do not intersect except at their common endpoint.
We assume for simplicity that any two non-adjacent edges have only finitely many points in common, no two edges touch at an interior point, and no three edges meet at a common interior point.
Thus, every common interior point of two edges is a proper crossing.
As this paper does not consider anything other than drawings in the plane, we will sometimes omit the plane adjective.

\section{Construction---Proof of \cref{thm:tightness}}\label{sec:construct}

In this section, we provide a construction showing that the bound given in \cref{thm:hG} is tight up to lower-order terms for every feasible density $0 < \varepsilon \le \frac{1}{2}-\frac{1}{2n}$.

\lv{\thmtight*}

To prove \cref{thm:tightness}, we will use a construction, inspired by Ringel's following result on complete graphs \cite{ringel64}, where $DW_n$ is a planar drawing of the wheel graph $W_n$ on $n$ vertices; see \cref{fig-W6} and \cref{sec:prelim} for the definition of the wheel graph.

\begin {figure}%
\begin{center}
\begin{tikzpicture}[scale=0.8]\small
\coordinate (uu) at (0,0);
\tikzstyle{every node}=[draw, color=black, shape=circle, inner sep=1.3pt, fill=white]
\foreach \u in {18,90,...,378} { \node[thick] (W\u) at (\u:2) {}; }
\node[thick] at (uu) {};
\tikzstyle{every path}=[draw, color=black, semithick]
\begin{scope}[on background layer]
\foreach \u in {18,90,...,378} {   
        \draw (0,0)--(W\u);  \draw (uu)--(W\u);  \coordinate (uu) at (W\u);
}
\end{scope}[on background layer]
\end{tikzpicture}
\end{center}
\caption{The wheel graph on six vertices $W_6$; see \cref{sec:prelim} for the formal definition.}
\label{fig-W6}
\end{figure}

\begin{theorem}[\cite{ringel64}]\label{thm:DW}
    For every integer $n \geq 4$, we have $h(K_n) = 2n-2$. Moreover, if $D$ is a drawing of $K_n$ with $2n-2$ uncrossed edges, then $D$ contains the drawing $DW_n$ with all edges from $D \setminus DW_n$ being drawn in the outer face of $DW_n$. 
\end{theorem}

\subparagraph*{Construction of $G_{x,n}$ and $DG_{x,n}$.}\phantomsection\label{constr}
Given $n$ and $n-1\ge x\ge 3$, we construct $DG_{x,n}$ starting from a drawing of $K_{x+1}$ provided by \cref{thm:DW}, whose uncrossed edges form a drawing $DW_{x+1}$ of the wheel $W_{x+1}$.
Equivalently, this drawing is obtained from $DW_{x+1}$ by adding all missing edges between the $x$ vertices incident with the outer face; these outer vertices induce a copy of $K_x$, and together with the universal vertex they form $K_{x+1}$.
The remaining $n-x-1$ vertices are inserted to $DG_{x,n}$ in the following way:
Pick any triangle $T$ in the planar part of $DG_{x,n}$ and add a vertex inside this face.
Then add an edge between the new vertex and each of the three vertices forming $T$. Since these operations are performed inside triangular faces of the planar part, all newly added edges are uncrossed, and the planar part inside the outer cycle remains triangulated.
We apply this transformation until we obtain a graph with $n$ vertices.
Let $G_{x,n}$ denote the graph represented by the resulting drawing $DG_{x,n}$.
\cref{fig-W6} presents an example of $DW_6$ and \cref{fig-Gxn} provides an example of $DG_{5,11}$.
\begin{figure}
\begin{center}
\begin{tikzpicture}[scale=0.8]\small
\coordinate (uu) at (0,0);
\tikzstyle{every node}=[draw, color=black, shape=circle, inner sep=1.3pt, fill=white]
\foreach \u in {18,90,...,378} { \node[thick] (W\u) at (\u:2) {}; }
\foreach \u in {54,126,...,414} { \node[thick] (W\u) at (\u:1) {}; }
\node[thick] at (uu) {};
\tikzstyle{every path}=[draw, color=black, semithick]
\coordinate (W450) at (W90);
\begin{scope}[on background layer]
\foreach \u in {18,90,...,378} {   
        \draw (0,0)--(W\u);  \draw (uu)--(W\u);  \coordinate (uu) at (W\u);
}
\foreach \mid [evaluate=\mid as \next using int(\mid+36)] [evaluate=\mid as \prev using int(\mid-36)]in {54,126,...,414} {
       \draw (0,0)--(W\mid); \draw (W\mid)--(W\next); \draw (W\mid)--(W\prev);
}

\foreach \u [evaluate=\u as \v using {int(mod(\u+144,360))}] in {18,90,...,378}{
    \draw[dotted] (W\u) to [out=\u+90, in=\v-90, looseness = 1.6] (W\v);
}

\end{scope}[on background layer]
\end{tikzpicture}
\end{center}
\vspace{-3em}
\caption{An example of $DG_{5,11}$. %
Crossed edges are depicted using dotted lines. %
See the \hyperref[constr]{construction} paragraph for the formal description.}
\label{fig-Gxn}
\end{figure}

\begin{proof}[Proof of \cref{thm:tightness}]
We take $G_{x,n}$ as defined in the \hyperref[constr]{construction} paragraph for the appropriately chosen $x$ (based on $n$ and $\epsilon$), which is computed later in the proof.
The graph $G_{x,n}$ is connected, since it contains the connected graph $K_{x+1}$ and every added vertex is adjacent to three existing vertices.
Let $D$ be the subdrawing of $DG_{x,n}$ consisting of the wheel edges of $DW_{x+1}$ together with all edges added when inserting the remaining $n-x-1$ vertices.
Let $m'$ be the number of edges of $D$.
With the construction of $DG_{x,n}$, $m'$ is equal to the sum of the number of uncrossed edges in $DW_{x+1}$ and the number of added edges to obtain $G_{x,n}$.
We then have:
\begin{itemize}
    \item $m' = 2x + 3(n-x-1)$
    \item $m = \frac{x(x+1)}{2} + 3(n-x-1)$
\end{itemize}

Hence, we have $m' = 3n-3-x$ and  
\begin{align}\label{m-bound}
      m &= 3n-3 + \frac{x(x-5)}{2}.
\end{align}

Since $x\le n-1$, we have
\lv{\[}\sv{$}
6(n-1)-5x \ge n-1 \ge 0.
\sv{$}\lv{\]}
Therefore,
\begin{align}\label{eq:mx}
    2m = 6(n-1)+x(x-5) \ge x^2,
\end{align}
and hence $\sqrt{2m}\ge x$.
\lv{%

}%
Then, we combine \cref{eq:mx} with $m'\leq h(G_{x,n})$ and the following proves Property \hyperref[c:1]{1)} of \cref{thm:tightness}:

\begin{align}
\label{ineq:lower_hG}
3n-3-\sqrt{2m} \leq 3n-3-x \leq h(G_{x,n}). 
\end{align}

Our next step is to choose an appropriate value of $x$ to approach the desired density.
With $m = \varepsilon n^2$ in \cref{m-bound}, we obtain the following root for $x$:
\begin{align}\label{eq:root}
    x_0 = \frac{5}{2} + \sqrt{\left(\frac{5}{2}\right)^2 + 2 \left(\varepsilon n^2 - 3(n-1)\right)}.
\end{align}
Recall that $n \geq \frac{3}{\varepsilon}$ implies $\varepsilon n^2-3(n-1)\ge 3$, and hence $x_0\ge 6$.
The condition $\varepsilon \leq \frac{n-1}{2n}$ ensures that $x_0 \leq n-1$.
\lv{%

}%
Let $f(t)=3n-3+\frac{t(t-5)}{2}$.
Thus $f$ is increasing on $[x_0,\infty)$.
If we choose $x = \lceil x_0 \rceil$, then $x\ge 6$ and $x\le n-1$.
Hence, $x$ is admissible for the construction.
Since $x=\lceil x_0\rceil\ge x_0$, we have
\[
    m=f(x)\ge f(x_0)=\varepsilon n^2.
\]
Moreover, since $x<x_0+1$ and $f$ is increasing on $[x_0,\infty)$, we obtain
\begin{align*}
  m&=f(x) < f(x_0+1) = 3n-3+\frac{(x_0+1)(x_0-4)}{2} = 3n-3+\frac{x_0(x_0-5)}{2}+x_0-2 \\&= \varepsilon n^2+x_0-2 \le \varepsilon n^2+n-3.
\end{align*}
Therefore, we have:
\begin{align}\label{e-bound}
    \varepsilon \leq \frac{m}{n^2} < \varepsilon + \frac{1}{n} - \frac{3}{n^2}.
\end{align}

The proof follows by \cref{ineq:lower_hG} and \cref{e-bound}.
\end{proof}

\section{Proof of \cref{thm:hG}}\label{sec:main}

This section is dedicated to the proof of our main theorem.

\thmh*

We say that $D'$ is an \textit{uncrossed subdrawing} of a graph $G$ if there exists a drawing $D$ of $G$ such that $D'$ consists of all vertices of $D$ and a subset of the uncrossed edges of $D$. Note that $h(G)$ can be alternatively defined as the maximum number of edges of an uncrossed subdrawing of $G$.

The proof can be decomposed into three main steps. We first consider all faces up to a certain length $k-1$ of an uncrossed subdrawing with the maximum number of edges. By doing so, we obtain two bounds on the value of $h(G)$, \cref{lemma:simple_k} and \cref{lemma:complex_k}, defining an area of possible values. \cref{lemma:combined_bounds} then provides a bound over that area. We finish the proof by finding a good value for $k$.
We will use the following lemma from Balko, Hliněný, Masařík, Orthaber, Vogtenhuber, and Wagner \cite{Uncrossed2}:
\begin{lemma}[{\cite[Lemma 7]{Uncrossed2}}]
\label{lemma:7_balko}
    Let $D'$ be an uncrossed subdrawing of a connected graph $G$. Then $D'$ is a planar drawing and, for every edge $\{u,v\}$ of $G$, the vertices $u$ and $v$ are incident to a common face of $D'$. Moreover, there is an uncrossed subdrawing $D''$ of $G$ such that $D''$ represents a connected supergraph of the graph represented by $D'$.
\end{lemma}

Let $G$ be a connected graph with $n\ge 3$ vertices and $m$ edges.
Let $\tilde{D}$ be a drawing of $G$ such that $\tilde{D}$ contains $h(G)$ uncrossed edges.
Let $\tilde{D}'$ be the subdrawing of $\tilde{D}$ representing the subgraph of $G$ made of all uncrossed edges of $\tilde{D}$. Let $\tilde{m}$ be the number of uncrossed edges in $\tilde{D}$, which is also the number of edges of $\tilde{D}'$. We thus have $\tilde{m} = h(G)$. 
The subdrawing $\tilde{D}'$ is planar because its edges are uncrossed in $\tilde{D}$.
Moreover, by \cref{lemma:7_balko}, there is an uncrossed subdrawing $D''$ of $G$ representing a connected supergraph of the graph represented by $\tilde{D}'$.
Since $\tilde{D}'$ already has $h(G)$ edges, this supergraph cannot contain any additional edge.
Thus, $\tilde{D}'$ represents a connected plane graph.
Let $\mathcal{F}$ be the set of faces of $\tilde{D}'$, and $f = |\mathcal{F}|$. For each face $F$ of $\tilde{D}'$, $v(F)$ is the number of vertices of $\tilde{D}'$ incident with $F$, and $|F|$ is the length of the facial walk.
For every $k,\ell \geq 3$, we also define $\mathcal{F}_{\ell} = \{F \in \mathcal{F}, |F| = \ell\}$, $f_\ell = |\mathcal{F}_{\ell}|$, and \[\mathcal{F}_{\geq k} = \{F \in \mathcal{F}, |F| \geq k\} = \mathcal{F} \setminus \left(\bigcup_{\ell = 3}^{k-1} \mathcal{F}_{\ell}\right). \]

First, since $\tilde{D}'$ is connected, for each face $F$, we have $v(F) \leq |F|$, as we can count either one or both sides of each edge in $|F|$. If $\tilde{D}'$ is a tree, then $\tilde{m}=h(G)=n-1$, and the desired bound follows immediately from $m\le \binom{n}{2}$.
Thus, we may assume from now on that $\tilde{D}'$ is not a tree.
Thus, we can make the following assumption.
\begin{observation}
\label{asu:count_once}
Suppose $\tilde{D}'$ is not a tree. Then, for each face $F$ of $\tilde{D}'$, at least one edge $e$ on the facial walk of $F$ is counted only once in $|F|$. In particular, the facial walk of $F$ contains a cycle containing $e$ (hence, the two sides of this edge $e$ are on different faces).
\end{observation}

Note that for any face $F\in \mathcal{F}$, $v(F) \geq 3$, since by \cref{asu:count_once} the facial walk of $F$ contains a cycle, and $G$ is simple.
We also note the following simple fact.

\begin{observation}
\label{obs:sum_faces}
\sv{ $\sum_{F \in \mathcal{F}} |F| = 2\tilde{m}$.}
\lv{\begin{align*}
    \sum_{F \in \mathcal{F}} |F| = 2\tilde{m}.
\end{align*}
}
\end{observation}

We obtain a simple and useful bound on $\tilde{m}$ depending on the $f_\ell$ values directly from Euler's formula.
\sv{Note that the case $k=3$ corresponds to the trivial bound $\tilde{m} \leq 3n-6$.
  Its proof is standard and postponed to \Cref{sec:app}\footnote{We mark the statements whose proof is postponed to Appendix~\ref{sec:app} by $\clubsuit$ symbol.}.
}

\begin{lemma}[Simple bound\sv{ $\clubsuit$}] %
\label{lemma:simple_k}
For every $k \geq 3$, we have:

\begin{align*}
    \tilde{m} \leq \frac{k}{k-2}n - \frac{2k}{k-2} + \sum_{\ell = 3}^{k-1}\frac{k-\ell}{k-2} f_\ell.
\end{align*}
\end{lemma}

\lv{Note that the case $k=3$ corresponds to the trivial bound$\tilde{m} \leq 3n-6$.}

\toappendix{
  \begin{proof}[Proof of \cref{lemma:simple_k}]
Using \cref{obs:sum_faces}, we get: 

\begin{align*}
    \sum_{F \in \mathcal{F}_{\geq k}}|F| = 2\tilde{m}-\sum_{\ell = 3}^{k-1} \ell f_\ell.
\end{align*}

Every face in $\mathcal{F}_{\geq k}$ has by definition a size of at least $k$.

Thus, we deduce the following:

\begin{align}
\label{ineq:fi_k}
    k\left(f - \sum_{\ell=3}^{k-1} f_\ell\right) \leq 2\tilde{m}-\sum_{\ell = 3}^{k-1} \ell f_\ell.
\end{align}

By plugging Euler's formula $n - \tilde{m} + f = 2$ in \cref{ineq:fi_k}, we get the final bound: 

\begin{align*}
    \tilde{m} \leq \frac{k}{k-2}n - \frac{2k}{k-2} + \sum_{\ell = 3}^{k-1}\frac{k-\ell}{k-2} f_\ell. &\qedhere
\end{align*}
\end{proof}
}

We also make use of the following condition on the $f_\ell$ values:

\begin{lemma}
    
\label{lemma:ineq_sum_fl}
For every $k \geq 4$, 
    \begin{align*}
        \sum_{\ell = 3}^{k-1} (\ell - 2) f_\ell \leq 2n - 4.
    \end{align*}
  \end{lemma}

\begin{proof}
By \cref{asu:count_once}, every face has length at least $3$ and at most $2n-3$.
We can express $\tilde{m}$ and $f$ in terms of values of $f_\ell$: 
$\tilde{m} = (\sum_{\ell = 3}^{2n-3} \ell f_\ell)/2$
and 
    $f = \sum_{\ell = 3}^{2n-3} f_\ell$.
We can thus rewrite Euler's formula $n - \tilde{m} + f = 2$ as:
\begin{align}
\label{eq:euler_inf}
    \sum_{\ell = 3}^{2n-3} (\ell - 2) f_\ell = 2n - 4.
\end{align}

Since for any $\ell \geq 3$, $(\ell - 2) f_\ell \geq 0$, \cref{eq:euler_inf} gives the conclusion. 
\end{proof}

We then analyze a more complex bound, which is a generalization of a bound presented by Balko, Hliněný, Masařík, Orthaber, Vogtenhuber, and Wagner~\cite{Uncrossed2}: 
\begin{lemma}[Complex bound]
\label{lemma:complex_k}
Let $G$ be a connected graph with $n$ vertices and $m$ edges, where $m\ge n-1 \ge 2$.
For every integer $k\ge 3$, we have
\begin{equation*}
    b_k^-(n,m,f_3,\dots,f_{k-1}) \leq \tilde{m} \leq b_k^+(n,m,f_3,\dots,f_{k-1})
\end{equation*}
where
\begin{align*}
  &b_k^-(n,m,f_3,\dots,f_{k-1}) = \frac{3n - 5 +\sum_{\ell=3}^{k-1}\left(3-\frac{\ell}{2}\right)f_\ell - \sqrt{\Delta_k}}{2}, \\
  &b_k^+(n,m,f_3,\dots,f_{k-1}) = \frac{3n - 5 +\sum_{\ell=3}^{k-1}\left(3-\frac{\ell}{2}\right)f_\ell + \sqrt{\Delta_k}}{2},\\
  &\Delta_k = \left(3n-7-\frac{3}{2}\sum_{\ell = 3}^{k-1}(\ell-2)f_\ell \right)^2- 4 \left(m - 3n + 6 -\sum_{\ell = 3}^{k-1}\frac{(\ell-2)(\ell-3)}{2} f_\ell \right).
\end{align*}

If $m \leq 3n - 6 + \sum_{\ell = 3}^{k-1} \frac{(\ell-2)(\ell-3)}{2} f_\ell$, then $\Delta_k$ is nonnegative.

Otherwise, the following must hold:
\begin{align}
\label{ineq:complex_sum_faces}
\sum_{\ell = 3}^{k-1}(\ell-2) f_\ell \leq 2n - \frac{14}{3} - \frac{4}{3}\sqrt{m - 3n + 6 - \sum_{\ell = 3}^{k-1} \frac{(\ell-2)(\ell-3)}{2} f_\ell}.
 \end{align}   
\end{lemma}

Note that with $k = 3$, $b_3^+(n,m) = \frac{3n-5 + \sqrt{(3n-5)^2-4m}}{2}$, which is the bound presented in~\cite{Uncrossed2}.

\begin{proof}
We rewrite \cref{obs:sum_faces} as follows:
\begin{align}
\label{sum_faces_length_k}
    \sum_{F\in\mathcal{F}_{\geq k}} |F| = 2 \tilde{m} - \sum_{\ell = 3}^{k-1} \ell f_\ell.
\end{align}

Then, for each face $F$, there are $\binom{v(F)}{2}$ pairs of vertices incident with that face.
The subgraph of $\tilde{D}'$ traced by the facial walk of $F$ is connected on these $v(F)$ vertices, and by \cref{asu:count_once} it contains a cycle.
Hence, this subgraph has at least $v(F)$ edges, so at least $v(F)$ of the above pairs already form edges of $\tilde{D}'$.

Every edge of $\tilde{D}$ that is not in $\tilde{D}'$ has its endpoints incident with a common face of $\tilde{D}'$ by \cref{lemma:7_balko}. Moreover, since the edges of $\tilde{D}'$ are uncrossed in $\tilde{D}$, such an edge cannot cross the boundary of any face of $\tilde{D}'$. Hence, it is contained in a single face of $\tilde{D}'$.
For each face $F\in\mathcal{F}$, let $e_F$ be the number of edges of $E(G)\setminus E(\tilde{D}')$ drawn in $F$.
By the previous paragraph, each edge in $E(G)\setminus E(\tilde{D}')$ lies in a unique face of $\tilde{D}'$, so it is counted exactly once.
Thus,
\begin{align}
\label{ineq:vf}
    m-\tilde{m}
    = \sum_{F\in\mathcal{F}} e_F
    \leq \sum_{F \in \mathcal{F}} \left(\binom{v(F)}{2} - v(F)\right).
\end{align}

Since $3 \leq v(F) \leq |F| $ for every face $F\in \mathcal{F}$, we deduce the following inequality from \cref{ineq:vf}.

\begin{align}\label{eq:morePrecise}
    m - \tilde{m} \leq \frac{1}{2} \sum_{F \in \mathcal{F}}|F|(|F|-3) = \frac{1}{2}\left(\sum_{F \in \mathcal{F}_{\geq k}} |F| (|F|-3) + \sum_{\ell = 3}^{k-1} \ell(\ell-3) f_\ell\right)
\end{align}

We bound the right-hand side of \cref{eq:morePrecise} from above by bounding the sum over products by the product over sums and using \cref{sum_faces_length_k}.
Since all terms $|F|$ and $|F|-3$ are nonnegative, we use
\[
    \sum_{F\in\mathcal{F}_{\ge k}} |F|(|F|-3)
    \leq
    \left(\sum_{F\in\mathcal{F}_{\ge k}} |F|\right)
    \left(\sum_{F\in\mathcal{F}_{\ge k}} (|F|-3)\right).
\]
Thus, we obtain:

\begin{align}\label{eq:rightHand}
    \frac{1}{2} \left[\left(2\tilde{m} - \sum_{\ell = 3}^{k-1} \ell f_\ell\right)\left(2\tilde{m} - 3f - \sum_{\ell = 3}^{k-1} (\ell-3) f_\ell\right) + \sum_{\ell = 3}^{k-1} \ell(\ell-3) f_\ell\right].
\end{align}

We use Euler's formula $n - \tilde{m} + f = 2$ to substitute $f$ and obtain the following polynomial inequality based on \cref{eq:morePrecise} with \cref{eq:rightHand}:

\begin{align}
  \sv{&}  \tilde{m}^2 - \left(3n-5 + \sum_{\ell = 3}^{k-1} \left(3-\frac{\ell}{2}\right)f_\ell\right)\tilde{m} + m - \frac{1}{2}\left(\sum_{\ell=3}^{k-1}\ell f_\ell\right)\left(\sum_{\ell=3}^{k-1}(\ell-3) f_\ell\right) \sv{\nonumber\\& }+ \frac{1}{2}\sum_{\ell=3}^{k-1}\left(3(n-1) - \ell\right)\ell f_\ell\ \leq 0.
\label{ineq:mi}
\end{align}

The discriminant of this inequality, solving for $\tilde{m}$, is:
\begin{align*}
  \Delta_k \sv{&}= \left(3n - 5 +\sum_{\ell=3}^{k-1}\left(3-\frac{\ell}{2}\right)f_\ell\right)^2 \sv{\\&}- 4\left(m - \frac{1}{2}\left(\sum_{\ell=3}^{k-1}\ell f_\ell\right)\left(\sum_{\ell=3}^{k-1}(\ell-3) f_\ell\right) + \frac{1}{2}\sum_{\ell=3}^{k-1}\left(3(n-1) - \ell\right)\ell f_\ell\right).
\end{align*}

We can rewrite the previous as:%
\begin{align}%
    \Delta_k = \left(3n-7-\frac{3}{2}\sum_{\ell = 3}^{k-1}(\ell-2)f_\ell\right)^2- 4 \left(m - 3n + 6 -\sum_{\ell = 3}^{k-1}\frac{(\ell-2)(\ell-3)}{2} f_\ell \right).
\end{align}

Note that $\Delta_k$ must be nonnegative, otherwise \cref{ineq:mi} cannot hold. Consequently, solving the equation $\Delta_k \geq 0$ gives us the following conditions.

    If $m \leq 3n-6 + \sum_{\ell = 3}^{k-1} \frac{(\ell-2)(\ell-3)}{2}f_\ell$, then $\Delta_k$ is nonnegative.

    If $m > 3n-6 + \sum_{\ell = 3}^{k-1} \frac{(\ell-2)(\ell-3)}{2}f_\ell$, then $\Delta_k$ is nonnegative if and only if 
    \begin{align}\label{eq:disc-positive}\sum_{\ell = 3}^{k-1}(\ell-2) f_\ell \leq 2n - \frac{14}{3} - \frac{4}{3}\sqrt{m - 3n + 6 - \sum_{\ell = 3}^{k-1} \frac{(\ell-2)(\ell-3)}{2} f_\ell}\end{align}
    or \[\sum_{\ell = 3}^{k-1}(\ell-2) f_\ell \geq 2n - \frac{14}{3} + \frac{4}{3}\sqrt{m - 3n + 6 - \sum_{\ell = 3}^{k-1} \frac{(\ell-2)(\ell-3)}{2} f_\ell}.\]
    However, since $m - 3n+6 - \sum_{\ell = 3}^{k-1} \frac{(\ell-2)(\ell-3)}{2}f_\ell > 0$ and it is an integer, it is at least one. Hence, \begin{align}\label{eq:contr}2n - \frac{14}{3} + \frac{4}{3}\sqrt{m - 3n + 6 - \sum_{\ell = 3}^{k-1} \frac{(\ell-2)(\ell-3)}{2} f_\ell} \geq 2n - \frac{10}{3} > 2n-4.\end{align}
    
    The alternative in \cref{eq:contr} is not possible by \cref{lemma:ineq_sum_fl}.
    Therefore, \cref{eq:disc-positive}
    must hold.
Since the left-hand side of \cref{ineq:mi} is a quadratic polynomial in $\tilde{m}$ with leading coefficient $1$, solving the inequality yields precisely the interval between its two roots, namely $b_k^- \leq \tilde{m} \leq b_k^+$.
\end{proof}

We now combine \cref{lemma:simple_k} and \cref{lemma:complex_k} to prove another bound that does not depend on the $f_\ell$ values anymore.
\sv{\cref{fig:asymptotical_intersection} presents a visual representation of the intersection between those bounds, with the filled area showing the possible values of $\frac{\tilde{m}}{n}$.
The maximum possible value of $\frac{\tilde{m}}{n}$ corresponds to the multiplicative constant $c_{\varepsilon}$ described in the discussion after \cref{thm:unc}.}

\begin{lemma}[Combined Bound\sv{ $\clubsuit$}]
\label{lemma:combined_bounds} 
Let $G$ be a connected graph with $n$ vertices and $m$ edges, and let $\tilde{m}=h(G)$.
For $k \geq 3$, if $m > (k-1)(n-2)$, then the following holds.
For $k=3$, we have the trivial bound
\[
    \tilde{m}\le 3n-6.
\]
For $k\ge4$, we have
\begin{align}\label{eq:final}
\tilde{m} \leq 3n - 7 + \frac{3}{k} -  (k-3)\sqrt{\frac{2\left(m-(n-2)(k-1)\right)}{k (k-3)} + \frac{1}{k^2}}.
\end{align}
\end{lemma}

\begin{figure}
    \centering
    \includegraphics[width=0.65\linewidth]{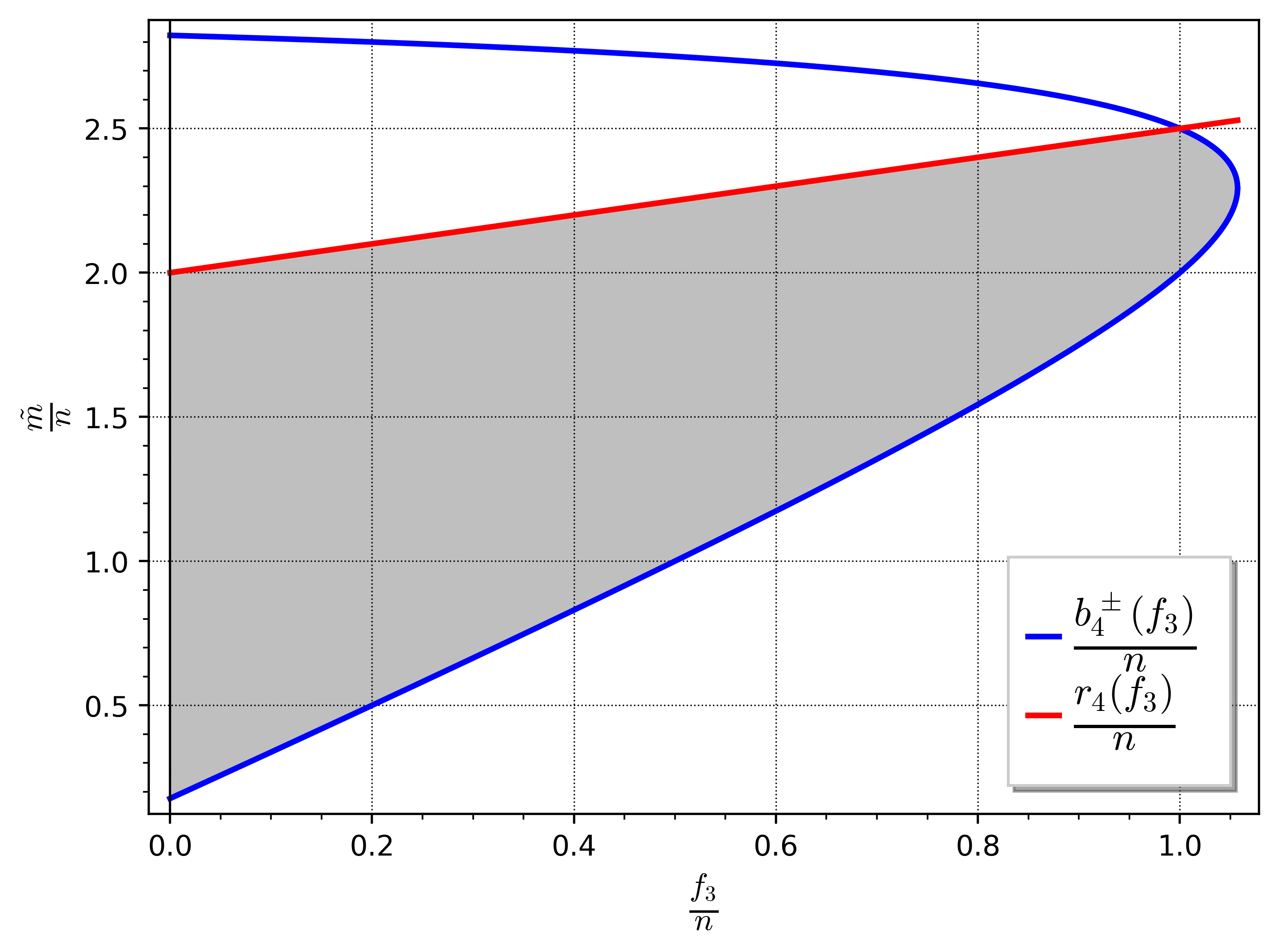}
    \caption{A visual representation of the asymptotic intersection between $\frac{r_4}{n}$ (red line) and $\frac{b_4^{\pm}}{n}$ for $\varepsilon = \frac{1}{2}$ (blue line).
      The gray filled area highlights the feasible values of $\frac{\tilde{m}}{n}$. 
      The maximum possible value of $\frac{\tilde{m}}{n}$ corresponds to the multiplicative constant $c_\epsilon$ described in the introduction.
    }
    \label{fig:asymptotical_intersection}
\end{figure}

\toappendix{
\lv{\begin{proof}}
\sv{\begin{proof}[Proof of \cref{lemma:combined_bounds}]}
For $k=3$, the statement is exactly the trivial planar bound $\tilde{m}\le 3n-6$.
We may therefore assume that $k\ge4$.

For $k \geq 4$, we combine the bounds from \cref{lemma:simple_k,lemma:complex_k} to find a bound on $\tilde{m}$ over all possible values of $f_3, \dots, f_{k-1}$.

Denote by $r_k\coloneqq \frac{k}{k-2}n - \frac{2k}{k-2} + \sum_{\ell = 3}^{k-1}\frac{k-\ell}{k-2} f_\ell$ the bound stated in \cref{lemma:simple_k}.
Fix admissible values of $f_4,\dots,f_{k-1}$, and view $r_k$, $b_k^-$, and $b_k^+$ as functions of $f_3$.
The function $r_k$ is increasing in $f_3$.
The feasible values of $\tilde{m}$ must satisfy both $\tilde{m}\le r_k$ and $b_k^-\le \tilde{m}\le b_k^+$.
Thus the maximum possible value is attained at an intersection of $r_k$ with one of $b_k^-$ and $b_k^+$.

Therefore, our task will be to find an intersection point between $r_k$ and one of $b_k^+$ or $b_k^-$, which is where $\tilde{m}$ can reach its highest value.
\cref{fig:asymptotical_intersection} presents a visual representation of the intersection between those bounds, with the filled area showing the possible values of $\frac{\tilde{m}}{n}$. The maximum possible value of $\frac{\tilde{m}}{n}$ corresponds to the multiplicative constant $c_{\varepsilon}$ described in the discussion after \cref{thm:unc}.

The equation to find the intersection between $b_k^+$ (resp. $b_k^-$) and $r_k$ %
is the following:

\begin{align}
  \frac{k-6}{k-2}n - \frac{k-10}{k-2} - \frac{k-6}{2(k-2)} \sum_{\ell = 3}^{k-1} (\ell -2)f_\ell = -\sqrt{\Delta_k}. \label{eq:startingEq}
\end{align}
(resp. $\frac{k-6}{k-2}n - \frac{k-10}{k-2} - \frac{k-6}{2(k-2)} \sum_{\ell = 3}^{k-1} (\ell -2)f_\ell = \sqrt{\Delta_k}$)

Recall that 
\[
    \Delta_k = \left(3n-7-\frac{3}{2}\sum_{\ell = 3}^{k-1}(\ell-2)f_\ell\right)^2- 4 \left(m - 3n + 6 -\sum_{\ell = 3}^{k-1}\frac{(\ell-2)(\ell-3)}{2} f_\ell \right).
  \]
Accordingly, for fixed $f_4,\dots,f_{k-1}$, we solve the equation:
\begin{align}
\label{eq:delta_k}
    \left(\frac{k-6}{k-2}n - \frac{k-10}{k-2} - \frac{k-6}{2(k-2)} \sum_{\ell = 3}^{k-1} (\ell -2)f_\ell\right)^2 = \Delta_k.
\end{align}

Observe that using \cref{lemma:ineq_sum_fl} we have: 
\begin{align}\label{eq:prev}
    \sum_{\ell = 3}^{k-1} \frac{(\ell-2)(\ell-3)}{2} f_\ell \leq (k-4)\left(\sum_{\ell = 3}^{k-1} \frac{\ell-2}{2} f_\ell\right) \leq (k-4)(n-2).
\end{align}
 
Hence, using the assumption $m > (k-1)(n-2)$ and \cref{eq:prev} we obtain:
\begin{align}
\label{ineq_m}
    m > 3n-6+(k-4)(n-2) \geq 3n -6 + \sum_{\ell = 3}^{k-1} \frac{(\ell-2)(\ell-3)}{2} f_\ell.
\end{align}

Solving \cref{eq:delta_k} for $\sum_{\ell = 3}^{k-1} (\ell -2)f_\ell$ then gives us the following roots:
\begin{align}\label{eq:newroots}
    \sum_{\ell = 3}^{k-1} (\ell -2)f_\ell = 2n - 5 + \frac{2}{k}\pm (k-2)\sqrt{\frac{2\left(m-3n+6-\sum_{\ell = 3}^{k-1} \frac{(\ell-2)(\ell-3)}{2} f_\ell\right)}{k (k-3)} + \frac{1}{k^2}}.
\end{align}

By \cref{ineq_m}, the quantity under the square root in \cref{eq:newroots} is always positive.
Recall \[\sum_{\ell = 3}^{k-1} (\ell -2)f_\ell \leq 2n -4\] by \cref{lemma:ineq_sum_fl}. 
Together with $k\geq 4$  we conclude that among the two roots in \cref{eq:newroots}, only the root with the minus sign can correspond to an admissible intersection.
Given $f_4, \dots, f_{k-1}$, let $f'_3$ denote the value of $f_3$ at this admissible intersection.
Since the term
\[
\sum_{\ell=3}^{k-1}\frac{(\ell-2)(\ell-3)}{2}f_\ell
\]
does not depend on $f_3$, we get:

\begin{align}\label{eq:theRoot}
    f'_3 = 2n - 5 + \frac{2}{k}- \sum_{\ell = 4}^{k-1}(\ell-2)f_\ell - (k-2)\sqrt{\frac{2\left(m-3n+6-\sum_{\ell = 4}^{k-1} \frac{(\ell-2)(\ell-3)}{2} f_\ell\right)}{k (k-3)} + \frac{1}{k^2}}.
\end{align}

Depending on the sign of (left hand-side of \cref{eq:startingEq})
\[
        \frac{k-6}{k-2}n - \frac{k-10}{k-2}
        - \frac{k-6}{2(k-2)}\sum_{\ell = 3}^{k-1} (\ell -2)f_\ell,
\]
$r_k$ intersects with $b_k^+$ if the value is nonpositive, or with $b_k^-$ if the value is nonnegative. 
We conclude in the following way:
\begin{itemize}
    \item First, suppose that the admissible intersection is with $b_k^-$.
    We have $f_3\le f'_3$ as otherwise the inequality $b_k^-\le\tilde{m}\le r_k$ does not hold.
    Thus, starting from the $r_k$ bound using~\cref{eq:theRoot}, we obtain:
    \begin{align}
        \tilde{m} &\leq r_k= \frac{k}{k-2}n - \frac{2k}{k-2} + \sum_{\ell = 3}^{k-1}\frac{k-\ell}{k-2} f_\ell \nonumber\\
            &\leq \frac{k}{k-2}n - \frac{2k}{k-2} + \frac{k-3}{k-2}f'_3 + \sum_{\ell = 4}^{k-1}\frac{k-\ell}{k-2} f_\ell \nonumber\\
        &= 3n - 7 + \frac{3}{k} -  (k-3)\sqrt{\frac{2\left(m-3n+6-\sum_{\ell = 4}^{k-1} \frac{(\ell-2)(\ell-3)}{2} f_\ell\right)}{k (k-3)} + \frac{1}{k^2}} - \sum_{\ell = 4}^{k-1}(\ell-3) f_\ell \nonumber\\
        &\leq 3n - 7 + \frac{3}{k} -  (k-3)\sqrt{\frac{2\left(m-(n-2)(k-1)\right)}{k (k-3)} + \frac{1}{k^2}} - \sum_{\ell = 4}^{k-1}(\ell-3) f_\ell \nonumber\\
        &\leq 3n - 7 + \frac{3}{k} -  (k-3)\sqrt{\frac{2\left(m-(n-2)(k-1)\right)}{k (k-3)} + \frac{1}{k^2}}.      \label{eq:combined-start}
    \end{align}
    Hence, we obtain the claimed bound (\cref{eq:final}) in this case.
  \item In the remaining case, the root $f'_3$ corresponds to the intersection of $r_k$ with $b_k^+$. We know that $r_k$ is increasing when $f_3$ is increasing. Moreover, we have:
    \begin{align*}
        \frac{\partial b_k^+}{\partial f_3} = \frac{3}{4}\left(1-\frac{6n-14-3\sum_{\ell = 3}^{k-1}(\ell-2)f_\ell}{\sqrt{\left(6n-14-3\sum_{\ell = 3}^{k-1}(\ell-2)f_\ell\right)^2 - 16\left(m-\left(3n-6+\sum_{\ell = 3}^{k-1}\frac{(\ell-2)(\ell-3)}{2} f_\ell\right)\right)}}\right).
    \end{align*}
    By \cref{ineq_m}, the quantity
    \[
    Q\coloneqq m-\left(3n-6+\sum_{\ell = 3}^{k-1}\frac{(\ell-2)(\ell-3)}{2}f_\ell\right)
    \]
    is positive. Moreover, by \cref{ineq:complex_sum_faces}, every feasible choice satisfies
    \[
    \sum_{\ell = 3}^{k-1}(\ell-2)f_\ell
    \le
    2n-\frac{14}{3}-\frac{4}{3}\sqrt Q.
    \]
    Hence
    \[
    A\coloneqq 6n-14-3\sum_{\ell = 3}^{k-1}(\ell-2)f_\ell \ge 4\sqrt Q>0,
    \]
    Therefore $\sqrt{A^2-16Q}<A$, and so
    \[
        1-\frac{A}{\sqrt{A^2-16Q}}< 0.
    \]
    The displayed derivative is therefore negative, and $b_k^+$ decreases when $f_3$ increases.

    If $f_3 \leq f'_3$, we can bound $\tilde{m}$ using the same computation leading to \cref{eq:combined-start}.
    If $f_3 > f'_3$, we have shown that $b_k^+(f_3) \leq b_k^+(f'_3)$.
    Therefore, we conclude by the bound in \cref{eq:combined-start}.\qedhere
\end{itemize}
\end{proof}
}

The strength of \cref{lemma:combined_bounds} depends on the choice of $k$, which in turn is constrained by $n$ and $m$.
We conclude with the following simple lemma, which selects a suitable value of $k$, thereby providing a bound independent of $k$.

\begin{lemma}
\label{lemma:alpha}
    Let $G$ be a connected graph with $n$ vertices and $m$ edges, and let $\tilde{m}=h(G)$.
    For any real $\alpha>0$, if $m \geq \frac{3n-6}{\alpha^2}$, then
    \begin{align*}
        \tilde{m} \leq 3n-6-(1-\alpha)\sqrt{2m}.
    \end{align*}
\end{lemma}

\begin{proof}
If $\alpha \geq 1$, then $3n-6 \leq 3n-6-(1-\alpha)\sqrt{2m}$, so the trivial bound $\tilde{m}\le 3n-6$ gives the claim.
Thus, we may assume that $\alpha < 1$.

Set $k \coloneqq \left\lceil\frac{3}{\alpha}\right\rceil$.
Since $\alpha<1$, we have $k\ge4$.
Then
\begin{align}\label{eq:choiceK}
    k-1 < \frac{3}{\alpha} \leq k.
\end{align}
By \cref{eq:choiceK} and the assumption of the lemma,
\[
m \geq \frac{3n-6}{\alpha^2} \geq \frac{(n-2)(k-1)}{\alpha}.
\]
Since $\alpha<1$, this also implies $m>(n-2)(k-1)$, so \cref{lemma:combined_bounds} applies.
It follows that
\begin{align}\label{eq:mAlpha}
    m - (n-2)(k-1) \geq (1-\alpha)m.
\end{align}

Using \cref{eq:mAlpha}, we have
\[
(k-3)\sqrt{\frac{2\left(m-(n-2)(k-1)\right)}{k(k-3)}+\frac{1}{k^2}}
\ge
\sqrt{1-\frac{3}{k}}\sqrt{2m(1-\alpha)}.
\]
Moreover, $3n-7+\frac{3}{k}\le 3n-6$.
Finally, $\frac{3}{\alpha}\le k$ gives $\frac{3}{k}\le\alpha$, and hence
\[
\sqrt{1-\frac{3}{k}}\sqrt{1-\alpha}\ge 1-\alpha.
\]

Starting from \cref{lemma:combined_bounds} and using the estimates above, we conclude:
\begin{align*}
    \tilde{m}
    &\leq 3n - 7 + \frac{3}{k}
    -  (k-3)\sqrt{\frac{2\left(m-(n-2)(k-1)\right)}{k (k-3)} + \frac{1}{k^2}} \\
    &\leq 3n - 6 - \sqrt{1-\frac{3}{k}} \sqrt{2m(1-\alpha)} \\
    &\leq 3n - 6 - (1-\alpha)\sqrt{2m}. \qedhere
\end{align*}
\end{proof}

The main theorem (\cref{thm:hG}) now follows from \cref{lemma:alpha}.

\begin{proof}[Proof of \cref{thm:hG}]
    Since $G$ is connected and $n\ge3$, we have $m\ge n-1>0$.
    Since we have $\tilde{m} = h(G)$, we apply \cref{lemma:alpha} with $\alpha = \sqrt{\frac{3n-6}{m}}$ to obtain the desired bound.
\end{proof}

\section{Stronger Bounds for Triangle-Free Graphs}\label{sec:triangle}
In this section, we prove two theorems for triangle-free graphs.
\subsection{Proof of \cref{thm:hG-triangle}}

\lv{\thmhtfree*}

We keep the notations from \cref{sec:main}.
Note that if $\tilde{D}'$ is a tree, then $\tilde{m}=n-1$.
Using Mantel's bound $m\le \lfloor n^2/4\rfloor$, the desired inequality
\[
n-1 \leq 2n-4-\sqrt{\frac{m}{2}} + \sqrt{\frac{5}{2}(n-2)}
\]
follows directly.
Thus, we may assume that $\tilde{D}'$ is not a tree and use \cref{asu:count_once}.
Since $G$ is triangle-free, so is the subgraph represented by $\tilde{D}'$.
Moreover, by \cref{asu:count_once}, every face of length $3$ would contain a cycle of length $3$.
Hence $f_3=0$.
We now use lemmas proved in \cref{sec:main}, and the plan is to improve the combined bound from \cref{lemma:combined_bounds} for triangle-free graphs (i.e., when $f_3 = 0$).

\begin{lemma}[Combined Bound for Triangle-Free Graphs\sv{ $\clubsuit$}]
\label{lemma:combined_bounds_tfree} 
Let $G$ be a triangle-free connected graph with $n$ vertices and $m$ edges, and let $\tilde{m}=h(G)$.
For $k \geq 3$, if $m > (k-1)(n-2)$, then the following holds.
For $k\le 4$, we have the trivial triangle-free planar bound
\[
    \tilde{m}\le 2n-4.
\]
For $k\ge5$, we have
\begin{align}\label{eq:final_tfree}
\tilde{m} \leq 2n - \frac{9}{2} + \frac{2}{k} - (k-4)\sqrt{\frac{m-(n-2)(k-1)}{2k(k-3)} + \frac{1}{4k^2}}.
\end{align}
\end{lemma}

The proof is very similar to that of \cref{lemma:combined_bounds}.
\sv{We postpone the proof to \Cref{sec:app}.}
\toappendix{
  \begin{proof}[Proof of \cref{lemma:combined_bounds_tfree}]

For $k\le 4$, the statement is the trivial triangle-free planar bound $\tilde{m}\le 2n-4$.
Thus, we may assume that $k\ge5$.

For $k \geq 5$, we can combine bounds from \cref{lemma:simple_k} and \cref{lemma:complex_k} the same way we did in the proof of \cref{lemma:combined_bounds}. Once again, $r_k$ intersects with either $b_k^+$ or $b_k^-$ when \cref{eq:theRoot} is verified. Since $f_3 = 0$ for triangle-free graphs, the equation verified at this intersection becomes:

\begin{align}
\label{eq:theRoot_tfree}
    \sum_{\ell = 4}^{k-1} (\ell -2)f_\ell = 2n - 5 + \frac{2}{k}- (k-2)\sqrt{\frac{2\left(m-3n+6-\sum_{\ell = 4}^{k-1} \frac{(\ell-2)(\ell-3)}{2} f_\ell\right)}{k (k-3)} + \frac{1}{k^2}}.
\end{align}

Using \cref{lemma:ineq_sum_fl}, we have
\[
\sum_{\ell = 4}^{k-1}\frac{(\ell-2)(\ell-3)}{2}f_\ell
\le (k-4)(n-2).
\]
Together with $m>(k-1)(n-2)$, this gives
\[
m>3n-6+\sum_{\ell = 4}^{k-1}\frac{(\ell-2)(\ell-3)}{2}f_\ell.
\]

Depending on the sign of $\frac{k-6}{k-2}n - \frac{k-10}{k-2} - \frac{k-6}{2(k-2)} \sum_{\ell = 4}^{k-1} (\ell -2)f_\ell$, this solution corresponds either to the intersection of $r_k$ with $b_k^+$ (if it is $\leq0$) or with $b_k^-$ (if it is $\geq0$).
Given $f_5,\dots,f_{k-1}$, let $f'_4$ denote the value of $f_4$ at this admissible intersection.
We conclude in the following way, analogously to the proof of \cref{lemma:combined_bounds}:
\begin{itemize}
    \item If $r_k$ intersects $b_k^-$, then, as in the proof of \cref{lemma:combined_bounds}, the feasibility condition $b_k^-\le \tilde{m}\le r_k$ implies that the actual value of $f_4$ is at most the value at the admissible intersection.
      Thus, starting from the $r_k$ bound, we obtain:
    \begin{align}
        \tilde{m} &\leq r_k= \frac{k}{k-2}n - \frac{2k}{k-2} + \sum_{\ell = 4}^{k-1}\frac{k-\ell}{k-2} f_\ell \nonumber\\
            &\leq \frac{k}{k-2}n - \frac{2k}{k-2} + \frac{k-4}{k-2}f'_4 + \sum_{\ell = 5}^{k-1}\frac{k-\ell}{k-2} f_\ell \nonumber\\
        &= 2n - \frac{9}{2} + \frac{2}{k} - \frac{k-4}{2}\sqrt{\frac{2\left(m-(n-2)(k-1)\right)}{k (k-3)} + \frac{1}{k^2}} - \sum_{\ell = 4}^{k-1}\frac{\ell-4}{2} f_\ell \nonumber\\
        &\leq 2n - \frac{9}{2} + \frac{2}{k} - (k-4)\sqrt{\frac{\left(m-(n-2)(k-1)\right)}{2k (k-3)} + \frac{1}{4k^2}}     \label{eq:combined-start_tfree}
    \end{align}
    Hence, we obtain the claimed bound (\cref{eq:final_tfree}) in this case.
  \item In the remaining case, the root $f'_4$ corresponds to the intersection of $r_k$ with $b_k^+$. We know that $r_k$ is increasing when $f_4$ is increasing. Moreover, we have:
    \begin{align*}
        \frac{\partial b_k^+}{\partial f_4} = \frac{1}{2}\left(1-\frac{3 \left(3n-7-\frac{3}{2}\sum_{\ell = 4}^{k-1}(\ell-2)f_\ell\right) -2}{\sqrt{\left(3n-7-\frac{3}{2}\sum_{\ell = 4}^{k-1}(\ell-2)f_\ell\right)^2 - 4\left(m-\left(3n-6+\sum_{\ell = 4}^{k-1}\frac{(\ell-2)(\ell-3)}{2} f_\ell\right)\right)}}\right).
    \end{align*}
Set
\[
Q=m-\left(3n-6+\sum_{\ell = 4}^{k-1}\frac{(\ell-2)(\ell-3)}{2}f_\ell\right)
\]
and
\[
B=3n-7-\frac{3}{2}\sum_{\ell = 4}^{k-1}(\ell-2)f_\ell.
\]
By the analogue of the inequality above, we have $Q>0$.
Moreover, \cref{ineq:complex_sum_faces} gives $B\ge 2\sqrt{Q}$.
Since $Q$ is a positive integer, $Q\ge1$, and hence $B\ge2$.
Therefore
\[
3B-2\ge B\ge \sqrt{B^2-4Q},
\]
which implies that the displayed derivative is nonpositive.
Thus $b_k^+$ decreases when $f_4$ increases.
    If $f_4 \leq f'_4$, we can bound $\tilde{m}$ using the same computation leading to \cref{eq:combined-start_tfree}.
    If $f_4 > f'_4$, we have shown that $b_k^+(f_4) \leq b_k^+(f'_4)$.
    Therefore, we conclude by the bound in \cref{eq:combined-start_tfree}.\qedhere
\end{itemize}
\end{proof}
}
The next lemma is a triangle-free analogue of \Cref{lemma:alpha}\sv{, and its proof is postponed to \cref{sec:app}}.

\sv{\begin{lemma}[$\clubsuit$]}
\lv{\begin{lemma}}
\label{lemma:alpha_tfree}
    Let $G$ be a triangle-free connected graph with $n$ vertices and $m$ edges, and let $\tilde{m}=h(G)$.
    For any real $\alpha>0$, if $m \geq \frac{5n-10}{\alpha^2}$, then
    \begin{align*}
        \tilde{m} \leq 2n-4-(1-\alpha)\sqrt{\frac{m}{2}}.
    \end{align*}
\end{lemma}

\toappendix{
  \begin{proof}[Proof of \cref{lemma:alpha_tfree}]
If $\alpha \geq 1$, then $2n-4 \leq 2n-4-(1-\alpha)\sqrt{\frac{m}{2}}$, so the triangle-free planar bound $\tilde{m}\le 2n-4$ gives the claim.
Thus, we may assume that $\alpha < 1$.

Set $k \coloneqq \left\lceil\frac{5}{\alpha}\right\rceil$.
Since $\alpha<1$, we have $k\ge6$.
Then
\begin{align}\label{eq:choiceK_tfree}
    k-1 < \frac{5}{\alpha} \leq k.
\end{align}
By \cref{eq:choiceK_tfree} and the assumption of the lemma,
\[
m \geq \frac{5n-10}{\alpha^2} > \frac{(n-2)(k-1)}{\alpha}.
\]
Since $\alpha<1$, this implies $m>(n-2)(k-1)$, so \cref{lemma:combined_bounds_tfree} applies.
It follows that
\begin{align}\label{eq:mAlpha_tfree}
    m - (n-2)(k-1) \geq (1-\alpha)m.
\end{align}

Using
\[
m-(n-2)(k-1)\ge (1-\alpha)m,
\]
we have
\[
(k-4)\sqrt{\frac{m-(n-2)(k-1)}{2k(k-3)}+\frac{1}{4k^2}}
\ge
\sqrt{1-\frac{5}{k}+\frac{1}{k^2}}\sqrt{\frac{m(1-\alpha)}{2}}.
\]
Moreover, $2n-\frac{9}{2}+\frac{2}{k}\le 2n-4$.
Finally, $\frac{5}{\alpha}\le k$ gives $\frac{5}{k}\le\alpha$, and hence
\[
\sqrt{1-\frac{5}{k}+\frac{1}{k^2}}\sqrt{1-\alpha}\ge 1-\alpha.
\]

Starting from \cref{lemma:combined_bounds_tfree} and using \cref{eq:mAlpha_tfree}, we conclude: 

\begin{align*}
    \tilde{m} &\leq 2n - \frac{9}{2} + \frac{2}{k} -  (k-4)\sqrt{\frac{m-(n-2)(k-1)}{2k (k-3)} + \frac{1}{4k^2}} \\
    &\leq 2n - 4 -  \sqrt{1-\frac{5}{k}+\frac{1}{k^2}} \sqrt{\frac{m(1-\alpha)}{2}} \\
    &\leq 2n - 4 - (1-\alpha)\sqrt{\frac{m}{2}} \qedhere
\end{align*}
\end{proof}
}

\cref{thm:hG-triangle} follows from \cref{lemma:alpha_tfree}.

\begin{proof}[Proof of \cref{thm:hG-triangle}]
    Since $G$ is connected and $n\ge3$, we have $m\ge n-1>0$.
    Applying \cref{lemma:alpha_tfree} with $\alpha=\sqrt{\frac{5n-10}{m}}$, whose hypothesis holds with equality, gives
    \[
    \tilde{m}
    \le
    2n-4-\sqrt{\frac{m}{2}}
    +\sqrt{\frac{5n-10}{m}}\sqrt{\frac{m}{2}}
    =
    2n-4-\sqrt{\frac{m}{2}}+\sqrt{\frac{5}{2}(n-2)}.
    \]
    Since $\tilde{m}=h(G)$, this proves the theorem.
\end{proof}

\subsection{Construction---Proof of \cref{thm:tightness_tfree}}

\lv{\thmtighttfree*}
We provide a construction similar to the one presented in \cref{sec:construct}, based on results of Mengersen~\cite{germanPaper}.

\begin{figure}[tb]
\centering
\begin{center}
\begin{tikzpicture}[scale=0.8]\small
\tikzstyle{every node}=[draw, color=black, shape=circle, inner sep=4pt]
\foreach \u in {0,90,...,360} { \node[thick, fill=white] (W\u) at (\u:2) {}; }
\node[thick, fill=white] (W90) at (90:2) {$v$};
\foreach \u in {45,135,...,360} { \node[thick, fill=black] (W\u) at (\u:2) {}; }

\tikzstyle{every path}=[draw, color=black, thick]
\begin{scope}[on background layer]
\foreach \u [evaluate=\u as \next using {int(mod(\u+45, 360))}] in {0,45,...,360} {   
    \draw (W\u)--(W\next);
}
\foreach \u in {225,315} {
    \draw (W90)--(W\u);
}
\foreach \u [evaluate=\u as \v using {int(mod(\u+135,360))}] [evaluate=\u as \w using {int(mod(\u+225,360))}] in {0,180,270} {
       \draw[dotted] (W\u) to [out=\u+90, in=\v-90, looseness = 1.7] (W\v);
       \draw[dotted] (W\w) to [out=\w+90, in=\u-90, looseness = 1.7] (W\u);
}

\end{scope}[on background layer]
\end{tikzpicture}
\end{center}
\vspace{-3em}
\caption{An example of $DK_{4,4}$. See the \hyperref[constr_DKxx]{Construction of $DK_{x,x}$} paragraph for the formal description.}
\label{fig-DK_44}
\end{figure}

\subparagraph*{Construction of $DK_{x,x}$.}\phantomsection\label{constr_DKxx}
Given $x \geq 2$, we can construct a drawing $DK_{x,x}$ of $K_{x,x}$ such that $DK_{x,x}$ has exactly $h(K_{x,x})$ uncrossed edges, i.e. $3x - 2$, uncrossed edges by \cref{thm-germanPaper}. One way to construct $DK_{x,x}$ is to start with $x$ white vertices and $x$ black vertices on the outer face, alternating white and black vertices and connecting each vertex to the next, to form a cycle of length $2x$. 
We can then draw the remaining $x-2$ uncrossed edges by choosing a vertex $v$ and drawing the edges to $x-2$ vertices of opposite color (except for the two vertices to which $v$ is already connected) inside the cycle.
The remaining $x^2-(3x-2)$ edges can then be drawn in the outer face; they may cross one another, but they do not cross the first $3x-2$ edges.
See \cref{fig-DK_44} for an illustration.

\begin{figure}
\begin{center}
\begin{tikzpicture}[scale=0.8]\small
\tikzstyle{every node}=[draw, color=black, shape=circle, inner sep=4pt]
\foreach \u in {0,90,...,360} { \node[thick, fill=white] (W\u) at (\u:2) {}; }
\node[thick, fill=white] (W90) at (90:2) {$v$};
\foreach \u in {45,135,...,360} { \node[thick, fill=black] (W\u) at (\u:2) {}; }
\node[thick, fill=white] (W1) at (0,0) {};
\node[thick, fill=white] (W2) at (0,-0.7) {};
\node[thick, fill=white] (W3) at (0,-1.4) {};

\tikzstyle{every path}=[draw, color=black, thick]
\begin{scope}[on background layer]
\foreach \u [evaluate=\u as \next using {int(mod(\u+45, 360))}] in {0,45,...,360} {   
    \draw (W\u)--(W\next);
}
\foreach \u in {225,315} {
    \draw (W90)--(W\u);
}
\foreach \u [evaluate=\u as \v using {int(mod(\u+135,360))}] [evaluate=\u as \w using {int(mod(\u+225,360))}] in {0,180,270} {
       \draw[dotted] (W\u) to [out=\u+90, in=\v-90, looseness = 1.7] (W\v);
       \draw[dotted] (W\w) to [out=\w+90, in=\u-90, looseness = 1.7] (W\u);
}
\foreach \u in {1,2,3}{
   \draw (W225)--(W\u); \draw (W315)--(W\u);
}

\end{scope}[on background layer]
\end{tikzpicture}
\end{center}
\vspace{-3em}
\caption{An example of $DH_{4,11}$. See the \hyperref[constr_tfree]{Construction of $DH_{x,n}$} paragraph for the formal description.}
\label{fig-Hxn_tfree}
\end{figure}

\subparagraph*{Construction of $H_{x,n}$ and $DH_{x,n}$.}\phantomsection\label{constr_tfree}
Given $n$ and $\frac{n}{2}\ge x\ge 4$, the graph $H_{x,n}$ is constructed from $K_{x,x}$ by adding $n-2x$ new vertices, each adjacent to the same two arbitrarily selected vertices from one partite class of $K_{x,x}$.
See \cref{fig-Hxn_tfree} for an illustration.
Then, we construct $DH_{x,n}$ starting from $DK_{x,x}$.
The remaining $n-2x$ vertices are inserted to $DH_{x,n}$ in the following way:
Pick a quadrilateral $Q$ in the planar part of $DH_{x,n}$, inside the outer face, whose two opposite vertices are the selected vertices from the same partite class.
Place all $n-2x$ new vertices inside $Q$ and connect each of them to these two opposite vertices.
These new vertices are placed in the other partite class, so the graph remains bipartite; the edges can be drawn inside $Q$ so that the planar part is subdivided into quadrilateral faces.
Note that all faces in the complement of the outer face remain quadrilaterals after this transformation, and the graph remains bipartite.

\medskip

\sv{We postpone the proof of \cref{thm:tightness_tfree} to \cref{sec:app}.}

\toappendix{
\begin{proof}[Proof of \cref{thm:tightness_tfree}]
We take $H_{x,n}$ as defined in the \hyperref[constr_tfree]{construction} paragraph, where the value of $x$ will be chosen below.
By construction, $H_{x,n}$ has $n$ vertices.
The graph $H_{x,n}$ is connected and bipartite, and hence triangle-free.
Let $D$ be the subdrawing of $DH_{x,n}$ consisting of the $3x-2$ uncrossed edges of $DK_{x,x}$ together with all edges added when inserting the remaining $n-2x$ vertices.
Since the new vertices are placed inside a quadrilateral face of the planar part and connected to the two opposite boundary vertices, all newly added edges are uncrossed.
Let $m'$ be the number of edges of $D$.
By the construction of $DH_{x,n}$, $m'$ is equal to the sum of the number of uncrossed edges in $DK_{x,x}$ and the number of added edges to obtain $H_{x,n}$.
Thus,
\begin{itemize}
    \item $m' = 3x-2 + 2(n-2x)$,
    \item $m = x^2 + 2(n-2x)$.
\end{itemize}

Hence, we have $m' = 2n - 2 - x$ and  
\begin{align}\label{m-bound_tfree}
      m &= 2n + x(x-4).
\end{align}

Since $n \geq 2x$, we have
\[
    m-x^2=2n-4x\ge0.
\]
Therefore,
\begin{align}\label{eq:mx_tfree}
    \sqrt{m} = \sqrt{2n+x(x-4)}
    \geq x.
\end{align}

Combining \cref{eq:mx_tfree} with $m'\leq h(H_{x,n})$ proves Property \hyperref[ctfree:1]{1)} of \cref{thm:tightness_tfree}:

\begin{align}
\label{ineq:lower_hG_tfree}
2n-2-\sqrt{m} \leq 2n-2-x \leq h(H_{x,n}). 
\end{align}

We now choose an appropriate value of $x$ to approach the desired density.

Recall \cref{m-bound_tfree}.
Solving the equation $2n+x(x-4)=m=\varepsilon n^2$ yields the following relevant root:
\begin{align}\label{eq:root_tfree}
    x_0 = 2 + \sqrt{\varepsilon n^2 - 2n + 4}.
\end{align}
Recall that $n \geq \frac{2}{\varepsilon}$ implies $\varepsilon n^2-2n\ge0$, and hence $x_0\ge4$.
The condition $\varepsilon \leq \frac{1}{4}$ ensures that $x_0 \leq \frac{n}{2}$.

Choose the integer $x = \lfloor x_0 \rfloor$.
Since $x_0\ge4$ and $x\le x_0\le n/2$, we have $4\le x\le n/2$.
Thus $x$ is admissible for the construction.

Let $f(t)=2n+t(t-4)$.
Since $f$ is increasing for $t\ge4$ and $x=\lfloor x_0\rfloor\le x_0$, we have
\[
    m=f(x)\le f(x_0)=\varepsilon n^2.
\]
Moreover, since $x>x_0-1$ and $f$ is increasing for $t\ge x_0-1\ge3$, starting from \cref{m-bound_tfree}, we obtain:
\begin{align*}
  m &=2n + x(x-4) > 2n + (x_0-1)(x_0-1-4) = 2n + x_0(x_0-4) - 2 x_0 + 5 \sv{\\&}= \varepsilon n^2 - 2x_0 +5\lv{\\ &} \geq \varepsilon n^2 - n + 5.
\end{align*}

Therefore,
\begin{align}\label{e-bound_tfree}
    \varepsilon -\frac{1}{n} + \frac{5}{n^2} < \frac{m}{n^2} \leq \varepsilon,
\end{align}
which implies Property~\hyperref[ctfree:2]{2)} of \cref{thm:tightness_tfree}.

The theorem follows from \cref{ineq:lower_hG_tfree} and \cref{e-bound_tfree}, since the strict lower-bound in \cref{e-bound_tfree} implies the required weak inequality.
\end{proof}
}

\section{Conclusions}\label{sec:conc}
We presented a new general bound for the uncrossed number in \cref{thm:unc}.
Our bound is tight up to lower-order terms when formulated in terms of the maximum uncrossed subgraph number (\cref{thm:hG}).
For the uncrossed number itself, however, we do not know whether the bound is tight for any fixed constant density (apart from complete graphs).

\begin{question}\label{q}
Is the bound in \cref{thm:unc} tight for dense graphs for every constant density $0<\varepsilon<\frac{1}{2}$, i.e., when $m=\varepsilon n^2$?
\end{question}

For triangle-free graphs, the bound in \cref{thm:hG-triangle} is probably not tight.
However, the construction from \cref{thm:tightness_tfree} provides some insight into what a tight general bound on $h(G)$ for triangle-free graphs could look like.

\begin{conjecture}[Conjectured tight bound for $h(G)$ on triangle-free graphs]\label{con:triangle}
 For every $n \geq 3$, for every connected triangle-free graph $G$ with $n$ vertices and $m$ edges, we have
 \begin{equation}\label{eq:conj_tfree}
     h(G) \le 2n-4 - \sqrt{m} + \sqrt{2(n-2)}.
 \end{equation}
\end{conjecture}

The additional $\sqrt{2(n-2)}$ term ensures \cref{eq:conj_tfree} remains true for planar triangle-free graphs with $m = 2n - 4$ edges.

Lastly, we reiterate \cite[Conjecture 16]{Uncrossed2}. Resolving \cref{q} positively for some constant $0<\varepsilon<\frac{1}{2}$ would imply this conjecture.
Indeed, if the bound in \cref{thm:unc} is tight for some graph class $\mathcal{G}$ of constant (independent of $n$) density $\varepsilon<\frac{1}{2}$, then $\unc(G)$ has a different asymptotic denominator constant from the outerthickness for graphs $G\in\mathcal{G}$, since the latter satisfies the trivial lower-bound
\[
\theta_o(G)\ge \left\lceil\frac{m}{2n-3}\right\rceil.
\]

\begin{conjecture}[{\cite[Conjecture 16]{Uncrossed2}}]\label{con:outerthickness-gap}
  For every positive integer $k$, there exists a graph $G$ such that \sv{$\theta_o(G)-\unc(G) \geq k$.}
  \lv{  \[\theta_o(G)-\unc(G) \geq k.\]}
\end{conjecture}

\lv{\bibliographystyle{alphaurl}}
\sv{\bibliographystyle{plainurl}}%
\bibliography{lit.bib}

\sv{
  \newpage
  \appendix 
  \section{Appendix}\label{sec:app}
\appendixText
}

\end{document}